\documentclass[12pt, reqno]{amsart}
\usepackage{amsmath, amsthm, amscd, amsfonts, amssymb, graphicx, color}

\newtheorem{theorem}{Theorem}[section]
\newtheorem{lemma}[theorem]{Lemma}
\newtheorem{proposition}[theorem]{Proposition}
\newtheorem{corollary}[theorem]{Corollary}
\theoremstyle{definition}
\newtheorem{definition}[theorem]{Definition}

\theoremstyle{remark}
\newtheorem{remark}[theorem]{Remark}
\numberwithin{equation}{section}

\begin{document}
\begin{flushleft}
	Stefan Ivkovi\'{c}
\end{flushleft}
The Mathematical Institute of the Serbian Academy of Sciences and Arts, \\
p.p. 367, Kneza Mihaila 36, 11000 Beograd, Serbia,\\
Tel.: +381-69-774237 \\
\email{stefan.iv10@outlook.com}\\

\textbf{\begin{center}
		On various generalizations of  semi-$\mathcal{A}$-Fredholm operators
\end{center}}	
\title{On various generalizations of  semi-$\mathcal{A}$-Fredholm operators }

\vspace{20pt}

\textbf{Abstract} Starting from the definition of $\mathcal{A}$-Fredholm and semi-$\mathcal{A}$-Fredholm operator on the standard module over a unital $C^{*}$ algebra $\mathcal{A}$, introduced in \cite{MF} and \cite{I}, we construct various generalizations of these operators  and obtain several results as an analogue or a generalization of some of the results in \cite{BS}, \cite{BM}, \cite{DDj2}, \cite{KY}. Moreover, we study also non-adjointable semi-$\mathcal{A}$-Fredholm operators as a natural continuation of the work in \cite{IM} on non-adjointable $\mathcal{A}$-Fredholm operators and obtain an analogue or a generalization in this setting of the results in \cite{I}, \cite{S}.

\vspace{15pt}
\textbf{Keywords} Generalized $\mathcal{A}$-Fredholm operator, generalized $\mathcal{A}$-Weyl operator, semi-$\mathcal{A}$-$B$-Fredholm operator, non-adjointable semi-$\mathcal{A}$-Fredholm operator.

\vspace{15pt}
\textbf{Mathematics Subject Classification (2010)} Primary MSC 47A53; Secondary MSC 46L08.

\begin{abstract}
Abstract

\textbf{Keywords} Hilbert C*-modules, semi-Fredholm operators, compressions, Weyl spectrum\\
\textbf{Mathematics Subject Classification (2010)} Primary MSC 47A53; Secondary MSC 46L08\\	
\end{abstract}

\section{Introduction }
Various generalizations of Fredholm and Weyl operators have been considered in several papers, such as \cite{BS}, \cite{BM}, \cite{DDj2}, \cite{KY}.\\
In \cite{KY} K.W. Yang has introduced the following definition of generalized Fredholm operator from Banach space $X$ into a Banach space $Y:$ \\
An operator $T \in B(X,Y)$ is generalized Fredholm if $T(X)$ is closed in $Y$, and $\ker T  $ and Coker $ T$ are reflexive.\\
Then he has obtained several results concerning these operators such as:
\begin{theorem} \label{T05}  
	\cite[Theorem 5.3]{KY} If $S \in B(X,Y)  $ and $T \in B(Y,Z)  $ are generalized Fredholm and $TS$ has a closed range, then $TS$ is generalized Fredholm.
\end{theorem}
\begin{theorem} \label{T06}  
	\cite[Theorem 5.4]{KY} Suppose $S \in B(X,Y)  $ and $  T \in B(Y,Z)$ are range closed, and suppose $ TS \in B(X,Z) $ is generalized Fredholm. Then, \\
	$(i)$ $S$ is generalized Fredholm $\Leftrightarrow T$ is generalized Fredholm;\\
	$(ii)$ if $\ker T$ is reflexive, then both $S$ and $T$ are generalized Fredholm;\\
	$(iii)$ if Coker $S$ is reflexive, then both $S$ and $T$ are generalized Fredholm.
\end{theorem}
\begin{theorem} \label{T07}  
	\cite[Theorem 5.5]{KY} Let $T \in B(X,Y)  $ have a closed range. If there exist $S$ $S^{\prime} \in B(Y,X) $ with closed ranges such that $ST$ and $TS^{\prime}$ are generalized Fredholm, then $T$ is generalized Fredholm.
\end{theorem}
\begin{theorem} \label{T08}  
	\cite[Theorem 5.6]{KY} Let $T \in B(X,Y)  $ be range closed. Then, $T$ is generalized Fredholm $\Leftrightarrow T^{*}$ is generalized Fredholm.
\end{theorem}
In \cite{DDj2} Djordjevic has considered generalized Weyl operators. The class of these generalized Weyl operators acting from a Hilbert space $H$ into a Hilbert space $K$ and denoted by $\Phi_{0}^{g}(H,K)$, is defined as:    $\Phi_{0}^{g}(H,K)=\lbrace T \in L(H.K): \mathcal{R}(T)  $ is closed and $\dim \mathcal{N}(T)=\dim \mathcal{N}(T^{*})  \rbrace ,$ where $L(H,K)$ denotes the set of all bounded operators from $H$ into $K.$  If $T \in \Phi_{0}^{g}(H,K),$ then $\mathcal{N}(T)  $ and $\mathcal{N}(T^{*}),$ may be mutually isomorphic infinite-dimensional Hilbert spaces.\\
Then he proves the following theorem.
\begin{theorem} \label{T04}  
	\cite[Theorem 1]{DDj2}
	Let $H,$ $K$ and $M$ be arbitrary Hilbert spaces, $T \in \Phi_{0}^{g}(H,K), S\in \Phi_{0}^{g}(H,M)  $ and $ \mathcal{R}(ST)  $ is closed. Then $ST \in \Phi_{0}^{g}(H,M).$
\end{theorem}
In the proof of this theorem he applies well known Kato theorem.\\
\\
Finally, in \cite{BS}  and \cite{BM}. Berkani has defined $B$-Fredholm and semi-$B$-Fredholm operators in the following way:\\
Let $ T \in L(X)$ where $X$ is a Banach space. Then $T$ is said to be semi-$B$-Fredholm if there exists an $ n $ such that $Im T^{n}$ is closed and $T_{\mid_{Im T^{n}}}  $ is a semi-Fredholm operator viewed as an operator from $Im T^{n}  $ into $Im T^{n}.$ If $T_{\mid_{Im T^{n}}}  $ is Fredholm, then $T$ is said to be $B$-Fredholm.\\
He proves for instance the following statements regarding these new classes of operators:
\begin{proposition} \label{P02}  
	\cite[Proposition 2.1]{BS} Let $T \in L(X).$ If there exists an integer $n \in \mathbb{N}  $ such that $R(T^{n}))$ is closed and such that the operator $T_{n}$ is an upper semi-Fredholm (resp. a lower semi-Fredholm) operator, then $R(T^{m}))$ is closed, $T_{m}  $ is an upper semi-Fredholm (resp.a lower semi-Fredholm) operator, for each $m \geq n  .$ Moreover, if $T_{n} $ is a Fredholm operator, then $ T_{m} $ is a Fredholm operator and $ind(T_{m})=ind (T_{n})$ for each $m \geq n  .$
\end{proposition}
\begin{proposition} \label{P03}  
	\cite[Proposition 3.3]{BS} Let $ T \in L(X) $ be a $  -$ Fredholm operator and let $F$ be a finite rank operator. Then $T+F$ is a $B$-Fredholm operator and ind $(T+F)= ind (T).$
\end{proposition}
Now, Hilbert $C^{*}$-modules are natural generalization of Hilbert spaces when the field of scalars is replaced by a $C^{*}$-algebra. 

Fredholm theory on Hilbert $C^*$-modules as a generalization of Fredholm theory on Hilbert spaces was started by Mishchenko and Fomenko in \cite{MF}. They have elaborated the notion of a Fredholm operator on the standard module $H_{\mathcal{A}} $ and proved the generalization of the Atkinson theorem. Their definition of $\mathcal{A}$-Fredholm operator on $H_{\mathcal{A}}$ is the following:

\cite[Definition ]{MF} A (bounded $\mathcal{A}$ linear) operator $F: H_{\mathcal{A}} \rightarrow H_{\mathcal{A}}$ is called $\mathcal{A}$-Fredholm if\\
1) it is adjointable;\\
2) there exists a decomposition of the domain $H_{\mathcal{A}}= M_{1} \tilde{\oplus} N_{1} ,$ and the range, $H_{\mathcal{A}}= M_{2} \tilde{\oplus} N_{2}$, where $M_{1},M_{2},N_{1},N_{2}$ are closed $\mathcal{A}$-modules and $N_{1},N_{2}$ have a finite number of generators, such that $F$ has the matrix from 
\begin{center}
	$\left\lbrack
	\begin{array}{ll}
	F_{1} & 0 \\
	0 & F_{4} \\
	\end{array}
	\right \rbrack
	$
\end{center}
with respect to these decompositions and $F_{1}:M_{1}\rightarrow M_{2}$ is an isomorphism.\\
The notation $\tilde{ \oplus} $ denotes the direct sum of modules without orthogonality, as given in \cite{MT}.\\
\\
In \cite{I} we vent further in this direction and defined semi-$\mathcal{A}$-Fredholm operators on Hilbert $C^{*}$-modules. We investigated then and proved several properties of these new semi Fredholm operators on Hilbert $C^{*}$-modules as an analogue or generalization of the well-known properties of classical semi-Fredholm operators on Hilbert and Banach spaces.\\
The main idea with this paper was to go further in the direction of \cite{I},  \cite{MF} and to define generalized $\mathcal{A}$-Fredholm operators, generalized $\mathcal{A}$-Weyl operators and semi-$\mathcal{A}$-$B$-Fredholm operators on $ H_{\mathcal{A}} $ that would  be appropriate generalizations of the above mentioned classes of operators on Hilbert and Banach spaces defined by Yang, Djordjevic and Berkani. Moreover the purpose of this paper is to establish in this setting an analogue or a generalization of the above mentioned results concerning generalized Fredholm, generalized Weyl and semi-$B$-Fredholm operators on a Hilbert or a Banach space. More precisely, our Proposition \ref{P01} is an analogue of \cite[Theorem 1]{DDj2}, our Lemma 3.5 is an analogue of \cite[Theorem 5.3]{KY}, our Proposition 3.6 is an analogue of \cite[Theorem 5.4]{KY}, our Lemma 3.7 is analogue of \cite[Theorem 5.5]{KY}, our Proposition 5.2 is a generalization of \cite[Proposition 2.1]{BM} and our Theorem \ref{T03} is a generalization of \cite[Proposition 3.3]{BM}. \\

Next, in addition to adjointable $\mathcal{A}$-Fredholm operator, Mishchenko also considers in \cite{IM}  non adjointable $\mathcal{A}$-Fredholm operators on the standard module $l_{2}(\mathcal{A})  .$ In this paper, we go further in this direction and consider non adjointable semi-$\mathcal{A}$-Fredholm operators on $l_{2}(\mathcal{A})  .$ We establish some of the basic properties of these operators in terns of inner and external (Noether) decompositions and show that these operators are exactly those that are one sided invertible in $ B (l_{2}(\mathcal{A})) / K(l_{2}(\mathcal{A})),$ where $K(l_{2}(\mathcal{A}))  $ denotes the set of all compact operators on $ l_{2}(\mathcal{A}) $ in the sense of \cite{IM}. Then we prove that an analogue or a modified version of results in \cite{I}, \cite{S} hold when one considers these non adjointable semi-$\mathcal{A}$-Fredholm operators.\\

\section{Preliminaries }
In this section we are going to introduce the notation, and the definitions in \cite{I} that are needed in this paper.	Throughout this paper we let $\mathcal{A}$ be a unital $\mathrm{C}^{*}$-algebra, $H_{\mathcal{A}}$ be the standard module over $\mathcal{A}$ and we let $B^{a}(H_{\mathcal{A}})$ denote the set of all bounded , adjointable operators on $H_{\mathcal{A}}.$ We also let $ B(l_{2}(\mathcal{A}))$  denote the set of all $\mathcal{A}$-linear, bounded operators on the standard module $l_{2}(\mathcal{A}),$ but not necessarily adjointable. According to \cite[ Definition 1.4.1] {MT}, we say that a Hilbert $\mathrm{C}^*$-module $M$ over $\mathcal{A}$ is finitely generated if there exists a finite set $ \lbrace x_{i} \rbrace \subseteq M $  such that $M $ equals the linear span (over $\mathrm{C}$ and $\mathcal{A} $) of this set.\\
\begin{definition}\label{D05}  
	\cite[Definition 2.1]{I} Let $\mathrm{F} \in B^{a}(H_{\mathcal{A}}).$ 
	We say that $\mathrm{F} $ is an upper semi-{$\mathcal{A}$}-Fredholm operator if there exists a decomposition 
	$$H_{\mathcal{A}} = M_{1} \tilde \oplus N_{1} \stackrel{\mathrm{F}}{\longrightarrow}   M_{2} \tilde \oplus N_{2}= H_{\mathcal{A}}$$
	with respect to which $\mathrm{F}$ has the matrix\\
	\begin{center}
		$\left\lbrack
		\begin{array}{ll}
		\mathrm{F}_{1} & 0 \\
		0 & \mathrm{F}_{4} \\
		\end{array}
		\right \rbrack,
		$
	\end{center}
	where $\mathrm{F}_{1}$ is an isomorphism $M_{1},M_{2},N_{1},N_{2}$ are closed submodules of $H_{\mathcal{A}} $ and $N_{1}$ is finitely generated. Similarly, we say that $\mathrm{F}$ is a lower semi-{$\mathcal{A}$}-Fredholm operator if all the above conditions hold except that in this case we assume that $N_{2}$ ( and not $N_{1}$ ) is finitely generated.	
\end{definition}
Set
\begin{center}
	$\mathcal{M}\Phi_{+}(H_{\mathcal{A}})=\lbrace \mathrm{F} \in B^{a}(H_{\mathcal{A}}) \mid \mathrm{F} $ is upper semi-{$\mathcal{A}$}-Fredholm $\rbrace ,$	
\end{center}
\begin{center}
	$\mathcal{M}\Phi_{-}(H_{\mathcal{A}})=\lbrace \mathrm{F} \in B^{a}(H_{\mathcal{A}}) \mid \mathrm{F} $ is lower semi-{$\mathcal{A}$}-Fredholm $\rbrace ,$	
\end{center}
$\mathcal{M}\Phi(H_{\mathcal{A}})=\lbrace \mathrm{F} \in B^{a}(H_{\mathcal{A}}) \mid \mathrm{F} $ is $\mathcal{A}$-Fredholm operator on $H_{\mathcal{A}}\rbrace .$
\begin{remark}\cite{I}
	Notice that if $M,N$ are two arbitrary Hilbert modules $\mathrm{C}^{*}$-modules, the definition above could be generalized to the classes $\mathcal{M}\Phi_{+}(M,N)$ and $\mathcal{M}\Phi_{-}(M,N)$.\\
	Recall that by \cite[ Definition 2.7.8]{MT}, originally given in \cite{MF}, when $\mathrm{F} \in \mathcal{M}\Phi(H_{\mathcal{A}})     $ and 
	$$ H_{\mathcal{A}} = M_{1} \tilde \oplus {N_{1}}\stackrel{\mathrm{F}}{\longrightarrow} M_{2} \tilde \oplus N_{2}= H_{\mathcal{A}} $$
	is an $ \mathcal{M}\Phi    $ decomposition for $  \mathrm{F}   $, then the index of $\mathrm{F}$ is definited by index $ \mathrm{F}=[N_{1}]-[N_{2}] \in K(\mathcal{A})    $ where $[N_{1}]    $ and $ [N_{2}] $ denote  the isomorphism classes of $ N_{1}    $ and $ N_{2} $ respectively. „By \cite[ Definition 2.7.9]{MT}, the index is well defined and does not depend on the choice of $\mathcal{M}\Phi$ decomposition for $\mathrm{F}.$
\end{remark}
\begin{definition} \label{D06}  
	\cite[Definition 5.6]{I} Let $F \in \mathcal{M}\Phi_{+} (H_{\mathcal{A}}).$ We say that $ F \in {{\mathcal{M}\Phi}_{+}^{-}}^{\prime} (H_{\mathcal{A}})$ if there exists a decomposition $$H_{\mathcal{A}} = M_{1} \tilde \oplus {N_{1}}‎‎\stackrel{F}{\longrightarrow} M_{2} \tilde \oplus N_{2}= H_{\mathcal{A}} $$
	with respect to which
	\begin{center}
		$F=\left\lbrack
		\begin{array}{ll}
		F_{1} & 0 \\
		0 & F_{4} \\
		\end{array}
		\right \rbrack,
		$
	\end{center}
	where $F_{1}$ is an isomorphism, $N_{1}$ is closed, finitely generated and $N_{1} \preceq N_{2} .$ Similarly, we define the class ${\mathcal{M}\Phi_{-}^{+}}^{\prime} (H_{\mathcal{A}})$, only in this case $F \in \mathcal{M}\Phi_{-} (H_{\mathcal{A}})$, $N_{2}$ is finitely generated and $N_{2} \preceq N_{1} .$
\end{definition}
In \cite{S} we set $ \widehat{\mathcal{M} \Phi}_{+}^{-} (H_{\mathcal{A}})$  to be the space of all $\mathrm{F} \in B^{a}(H_{\mathcal{A}}) $ such that there exists a decomposition
$$H_{\mathcal{A}} = M_{1} \tilde \oplus {N_{1}}\stackrel{\mathrm{F}}{\longrightarrow}  M_{2} \tilde \oplus N_{2}= H_{\mathcal{A}},$$  
w.r.t. which $\mathrm{F}$ has the matrix 
$\left\lbrack
\begin{array}{ll}
\mathrm{F}_{1} & 0 \\
0 & \mathrm{F}_{4} \\
\end{array}
\right \rbrack
,$ 
where $\mathrm{F}_{1}$ is an isomorphism,  $ N_{1}  $  is finitely generated and such that there exist closed submodules  $ N_{2}^{\prime},N  $  where  $N_{2}^{\prime} \subseteq N_{2},N_{2}^{\prime} \cong N_{1} ,$ 
$H_{\mathcal{A}}= N \tilde{\oplus} N_{1} = N \tilde{\oplus} N_{2}^{\prime} \textrm { and  the projection onto } N \textrm{ along } N_{2}^{\prime}$ is adjointable.\\
\begin{definition} \label{D07}  
	\cite[Definition 4]{S} We set $\widehat{\mathcal{M} \Phi}_{-}^{+}(H_{\mathcal{A}})$ to be the set of all $ \mathrm{D} \in B^{a}  (H_{\mathcal{A}}) $ such that there exists a decomposition
	$$H_{\mathcal{A}} = M_{1}^{\prime} \tilde \oplus {N_{1}^{\prime}}\stackrel{\mathrm{D}}{\longrightarrow} M_{2}^{\prime} \tilde \oplus N_{2}^{\prime}= H_{\mathcal{A}} $$ 
	w.r.t. which $\mathrm{D}$ has the matrix 
	$\left\lbrack
	\begin{array}{ll}
	\mathrm{D}_{1} & 0 \\
	0 & \mathrm{D}_{4}  \\
	\end{array}
	\right \rbrack
	,$ where $\mathrm{D}_{1}$ is an isomorphism, $N_{2}^{\prime}$ is finitely generated and such that $H_{\mathcal{A}}=M_{1}^{\prime} \tilde \oplus N \tilde \oplus N_{2}^{\prime}$ for some closed submodule $N,$ where the projection onto $M_{1}^{\prime} \tilde \oplus N $ along $N_{2}^{\prime}$ is adjointable.
\end{definition}
\begin{definition}  \label{D08}  
	\cite[Definition 2]{IM} A bounded $\mathcal{A}$-operator $l_{2}(\mathcal{A})  \longrightarrow l_{2}(\mathcal{A}) $ is called a Fredholm $\mathcal{A}$-operator if there exists a bounded $\mathcal{A}$-operator $  $ such that\\
	$$\mathrm{\textbf{id}} -FG \in \mathcal{K}(l_{2}(\mathcal{A})), \mathrm{\textbf{id}} -GF \in \mathcal{K}(l_{2}(\mathcal{A})).$$
\end{definition}
\begin{definition} \label{D09}  
	\cite[Definition 3]{IM} We say that a bounded $\mathcal{A}$-operator $F:l_{2}^{\prime}(\mathcal{A}) \longrightarrow  l_{2}^{\prime \prime}(\mathcal{A})$ admits an inner (Noether) decomposition if there is a decomposition of the preimage and the image
	$ l_{2}^{\prime}(\mathcal{A})=M_{1} \oplus N_{1}, l_{2}^{\prime \prime}(\mathcal{A})=M_{2} \oplus N_{2} $
	where $C^{*}$-modules $ N_{1}  $ and $N_{2}  $ are finitely generated Hilbert $C^{*}$-modules, and if $F$ has the following matrix from
	$F=\left\lbrack
	\begin{array}{ll}
	F_{1} & F_{2} \\
	0 & F_{4} \\
	\end{array}
	\right \rbrack
	:$ 
	$ M_{1} \oplus N_{1}  \longrightarrow M_{2} \oplus N_{2}, $
	where $F_{1}: M_{1} \longrightarrow M_{2}$ is an isomorphism.
\end{definition}
\begin{definition}  \label{D10}  
	\cite[Definition 4]{IM} We put by definition $ index F=[N_{2}]-[N_{1}] \in K(\mathcal{A}).$
\end{definition}
\begin{definition} \label{D11}  
	\cite[Definition 5]{IM} We say that a bounded $\mathcal{A}$-operator $F:l_{2}^{\prime}(\mathcal{A}) \longrightarrow  l_{2}^{\prime \prime}(\mathcal{A})  $ admits an external (Noether) decomposition if there exist finitely generated $C^{*}$-modules $X_{1}  $ and $X_{2} $ bounded $\mathcal{A}$-operators $E_{2},E_{3}  $ such that the matrix operator\\
	$F_{0}=\left\lbrack
		\begin{array}{ll}
		F & E_{2} \\
		E_{3} & 0 \\
	\end{array}
	\right \rbrack
	:$ 
	$ l_{2}^{\prime}(\mathcal{A}) \oplus X_{1}  \longrightarrow l_{2}^{\prime \prime}(\mathcal{A}) \oplus X_{2}, $
	Is an invertible operator.
\end{definition}
\begin{definition} \label{D12}  
	\cite[Definition 6]{IM} We put by definition $index F=[X_{1}]-[X_{2}]  \in K(\mathcal{A}) .$
\end{definition}

\section{On generalized $\mathcal{A}$-Fredholm and $\mathcal{A}$-Weyl operators  }

%
\begin{definition} \label{D01}  
	Let $F \in B^{a} (H_{\mathcal{A}}) $\\
	1) We say that $F \in  {\mathcal{M}\Phi}^{gc}(H_{\mathcal{A}})  $  if $ImF$ is closed, $ker F$ and $ Im F^{\perp}$ are self-dual.\\
	2)	We say that $F \in  {\mathcal{M}\Phi}_{0}^{gc}(H_{\mathcal{A}})   $ is $Im F$ is closed and $ker F \cong Im F^{\perp}$ (here we do \underline{not} require self-duality of $ker F,$ $Im F^{\perp}) .$
\end{definition}
\begin{proposition}  \label{P01}  
	Let $F,D \in  {\mathcal{M}\Phi}_{0}^{gc}(H_{\mathcal{A}}) $ and suppose that $ImDF$ is closed. Then $DF \in  {\mathcal{M}\Phi}_{0}^{gc}(H_{\mathcal{A}}) .$
\end{proposition}
\begin{proof} 
	Since $Im DF$ is closed, by \cite[Theorem 2.3.3]{MT} there exists a closed submodule $X$ s.t. $Im D = Im DF \oplus X.$ Next, considering the map $D_{\mid_{Im F}}  $ and again using that $Im DF$ is closed, we have that $\ker D \cap Im F= \ker D_{\mid_{Im F}}$ is orthogonally complementable in $Im F$ by \cite[Theorem 2.3.3]{MT}, so $ Im F = W \oplus (\ker D \cap Im F)$ for some closed submodule $W.$ Now, since $\ker D \cap Im F \oplus W \oplus Im F^{\perp}=H_{\mathcal{A}} $ and $(\ker D \cap Im F) \subseteq \ker D,$ it follows that $ \ker D= (\ker D \cap Im F) \oplus (\ker D \cap (W \oplus Im F^{\perp})) .$ Set $M=\ker D \cap (W \oplus Im F^{\perp}) ,$ then $\ker D= (\ker D \cap Im F) \oplus M  .$ On $ \ker D^{\perp} ,D$  is an isomorphism from $ \ker D^{\perp}  $ onto $Im D.$ Let $S=(D_{\mid_{\ker D^{\perp}}})^{-1} .$ Then $P_{{\ker {D^{\perp}}_{\mid_{W}} }} $ is an isomorphism from $W$ onto $S(Im DF).$ Indeed, since $D_{\mid_{W}}  $ is injective and $D(W)=Im DF$ is closed, by Banach open mapping theorem $D_{\mid_{W}}  $ is an isomorphism onto $Im DF.$ This actually means that $DP_{{\ker {D^{\perp}}_{\mid_{W}} }}   $ is an isomorphism onto $Im DF,$ as $D_{\mid_{W}} = DP_{{\ker {D^{\perp}}_{\mid_{W}} }} .$ Since $D_{\mid_{S(Im D)}}$ is an isomorphism onto $Im DF,$ it follows that $P_{{\ker {D^{\perp}}_{\mid_{W}} }}   $ is an isomorphism onto $S(Im DF).$ Hence $\sqcap_{{{S(Im DF)}_{\mid}}_{W}}  $ is an isomorphism onto $S(Im DF),$ where $\sqcap_{{S(Im DF)}}  $ denotes the projection onto $S(Im DF)$ along $S(X).$ Therefore we get that $H_{\mathcal{A}}=W \tilde{\oplus} S(X) \tilde{\oplus} \ker D.$ Thus we have $$H_{\mathcal{A}}=W \tilde{\oplus} S(X) \tilde{\oplus} (\ker D \cap Im F) \tilde{\oplus} M=W \tilde{\oplus} (\ker D \cap Im F) \oplus Im F^{\perp}  .$$ This gives $ S(X) \tilde{\oplus} M \cong Im F^{\perp}  .$ On the other hand, by clasical arguments we have $\ker DF = \ker F \tilde{\oplus} R$ for some closed submodule $R$ isomorphic to $ \ker D \cap Im F.$ Therefore we get $ \ker DF  \cong (\ker F \oplus (\ker D \cap ImF))\cong Im F^{\perp} \oplus (\ker D \cap Im F) \cong S(X) \oplus M \oplus \ker D \cap Im F \cong S(X) \oplus \ker D \cong X \oplus Im D^{\perp} \cong Im DF. $
	(where $ \oplus $ denotes now the direct sum in the sense of \cite[Example 1.3.3]{MT} ).
\end{proof}
\begin{remark}
	This result is a generalization of \cite[Theorem 1]{DDj2}, however in our proof we do not apply Kato theorem. Indeed, our proof is also valid in the case when $F \in  {\mathcal{M}\Phi}_{0}^{gc} (M,N), D \in  {\mathcal{M}}_{0}^{gc} (N,K)$ where $M,N,K$ are arbitrary Hilbert $C^{*}$-modules over a unital $C^{*}$-algebra $\mathcal{A}  .$ Next, by our proof we also obtain easily a generalization of Harte's ghost theorem: 
\end{remark}
\begin{corollary}
	Let $F,D \in B^{a} (H_{\mathcal{A}}) $ and suppose that $Im F, Im D, Im DF$ are closed. Then $\ker F \oplus \ker D \oplus Im DF^{\perp} \cong Im D^{\perp} \oplus Im F^{\perp} \oplus \ker DF. $
\end{corollary} 
\begin{proof}
	We keep the notation from the previous proof. In that proof we have shown that $Im F^{\perp} \cong S(X) \oplus M.$ Moreover $D=\ker D \cap Im F \oplus M $ and $Im DF^{\perp} =Im D^{\perp} \oplus X.$ This gives 
	
	$$ \ker F \oplus \ker D \oplus Im DF^{\perp} \cong \ker F \oplus \ker D \oplus ImD^{\perp} \oplus X \cong $$
	$$ \ker F \oplus (\ker D \cap Im F) \oplus M \oplus Im D^{\perp} \oplus X \cong \ker DF \oplus M \oplus S(X) \oplus Im D^{\perp} \cong $$
	$$\ker DF \oplus Im F^{\perp} \oplus Im D^{\perp} $$
\end{proof}	
	The next results are inspired by results in \cite{KY}.
\begin{lemma} \label{L01}  
	Let  $F,D \in  {\mathcal{M}\Phi}^{gc}(H_{\mathcal{A}})$ and suppose that $Im DF$ is closed.
Then $DF \in  {\mathcal{M}\Phi}^{gc}(H_{\mathcal{A}}).$ 
\end{lemma}
\begin{proof}
Suppose that $DF \in  {\mathcal{M}\Phi}^{gc}(H_{\mathcal{A}}).$ Then $ \ker F, \ker D$ are self-dual and $Im F,$ $ Im D$ are closed. Now, $D_{\mid_{Im F}}$ is an operator onto $Im DF=Im D_{\mid_{Im F}}$  which is closed by assumption and it is adjointable as $D$ is so and $Im F$ is orthogonally complemntable by \cite[Theorem 2.3.3]{MT}. Hence, again by \cite[Theorem 2.3.3]{MT} we deduce that $\ker D_{\mid_{Im F}} = \ker D \cap Im F $ is orthogonally complementable in $Im F,$ so $Im F= (\ker D \cap Im F) \oplus M$ for some closed submodule $M.$ Therefore $H_{\mathcal{A}}=(\ker D \cap Im F) \oplus M \oplus Im F^{\perp}.$ It follows that $\ker D=(\ker D \cap Im F) \oplus  M^{\prime}$ where $M^{\prime}= \ker D \cap (M \oplus Im F^{\perp}).$  On the other hand by classical arguments, one can show that $\ker DF= \ker F \tilde{\oplus} W$ where $W \cong \ker D \cap Im F.$ Since $ \ker F$ is self dual, $ \ker F$  is therefore an orthogonal direct summand in $ \ker DF$  by \cite[Proposition 2.5.4]{MT}, so $ \ker DF=\ker F \oplus \tilde{W}$  for some closed submodule $\tilde{W} \cong W \cong \ker D \cap Im F.$ Since $\ker D \cap Im F $ is self-dual, so is $\tilde{W} ,$ hence, $ \ker DF$ is self-dual being orthogonal direct sum of two self-dual modules. \\
Next, from the proof of Proposition \ref{P01} we obtain that $Im DF^{\perp}=Im D^{\perp} \oplus X,$ where $Im F^{\perp} \cong X \oplus M.$ Since $Im F^{\perp}$ is self-dual, so is $X$ being an orthogonal direct summand in a self dual module. Finally since $Im D^{\perp}$ is self-dual, it follows that $Im DF^{\perp}=Im D^{\perp} \oplus X$ is self-dual also. 
\end{proof}
\begin{proposition} \label{P05}  
	Let $F,D \in B^{a}(H_{\mathcal{A}}),$ suppose that $ Im F, Im D$ are closed and $Im DF \in {\mathcal{M}\Phi}^{gc}(H_{\mathcal{A}}).$ Then the folloving statements hold:\\
	a)	$D \in {\mathcal{M}\Phi}^{gc}(H_{\mathcal{A}}) \Leftrightarrow F \in {\mathcal{M}\Phi}^{gc}(H_{\mathcal{A}})$\\
	b)	if $ \ker D$ is self-dual then $F,D \in {\mathcal{M}\Phi}^{gc}(H_{\mathcal{A}}) $\\
	c) if $Im F^{\perp}$ is self-dual, then $F,D \in {\mathcal{M}\Phi}^{gc}(H_{\mathcal{A}}) .$
\end{proposition}
\begin{proof}
Let us prove b) first. If $DF$ is generalized $\mathcal{A}$-Fredholm, then $Im DF$ is closed and $ Im DF^{\perp}, \ker DF$ are self-dual. Now, observe that $Im DF = Im D_{\mid_{Im F}} = Im P_{Im D} D_{\mid_{Im F}} $ where $P_{Im D}$ denotes the orthogonal projection onto $Im D.$ Since $P_{Im D} D_{\mid_{Im F}}$ is adjointable, by \cite[Theorem 2.3.3]{MT}, we have that $Im DF$ is orthogonally complementable in $Im D.$ Hence $Im D= Im DF \oplus N$ for some closed submodule $N.$ Therefore $H_{\mathcal{A}}= Im DF \oplus N \oplus Im D^{\perp},$ so $ImDF^{\perp}=N \oplus Im D^{\perp}.$ Since $ Im DF^{\perp} $ is self-dual, so is $Im D^{\perp} ,$ being an orthogonal direct summand in $ Im DF^{\perp}.$ Next, since $ F(\ker DF)= \ker D \cap Im F$ and $F_{\mid_{\ker DF}}$ is adjointable, as $F$ is so and $ \ker DF$ is orthogonally complementable by \cite[Theorem 2.3.3]{MT}, we deduce that $ \ker F = \ker F_{\mid_{\ker DF}}$ orthogonally complementable in $\ker DF.$ Since $\ker DF$ is self-dual, it follows that $\ker F$ is self-dual, being orthogonal direct summand in $\ker DF.$ It remains to show that $Im F^{\perp}$ is self-dual. But, by earlier arguments, since $Im DF$ is closed, we have the $\ker D \cap Im F$ is orthogonally complementable $Im F,$ hence in $H_{\mathcal{A}}$ as $H_{\mathcal{A}}=Im F \oplus Im F^{\perp},$  and therefore in $\ker D.$ So $\ker D=( \ker D \cap Im F) \oplus M^{\prime}$ for some closed submodule $M^{\prime}.$ Moreover, again by arguments, we have then that $Im F^{\perp} \cong N \oplus M^{\prime}. $ Now, $N$ and $M^{\perp}$ are self dual, being orthogonal direct summands in $Im DF^{\perp}$ and $\ker D ,$ respectively, which are self-dual. Hence $M^{\prime} \oplus N$ is self-dual, thus $Im F^{\perp} $ is self-dual. By passing to the adjoints one may obtain c). To deduce a), use b) and c).	
\end{proof}
\begin{lemma} \label{L09}  
	Let $F \in B^{a}(H_{\mathcal{A}})   $ and supppose that $Im F$ is closed. Moreover, assume that there exist operators $D,D^{\prime} \in B^{a}(H_{\mathcal{A}})  $ with closed images such that $D^{\prime}F,FD \in \mathcal{M}\Phi^{gc} (H_{\mathcal{A}}) .$ Then $ F \in \mathcal{M}\Phi^{gc} (H_{\mathcal{A}}) .$
\end{lemma}
\begin{proof}
	By the proof of Proposition \ref{P05}, part b), since $Im FD$ is in $\mathcal{M}\Phi^{gc} (H_{\mathcal{A}})  $ and $Im F, Im D$ are closed, it follows that $Im F^{\perp}  $ is self-dual. Now, by passing ro the adjoints we obtain that $F^{*}(D^{\prime})^{*} \in \mathcal{M}\Phi^{gc} (H_{\mathcal{A}})$ as $D^{\prime}F \in  \mathcal{M}\Phi^{gc} (H_{\mathcal{A}}).$ Moreover, by the proof of \cite[Theorem 2.3.3]{MT} part ii), $Im F^{*}, (Im D^{\prime})^{*} $  are closed, as $Im F,Im D^{\prime}  $ are so (by assumption). Hence, using the previous arguments, we deduce that $Im {F^{*}}^{\perp} =\ker F  $ is self-dual.
\end{proof}

\section{Remarks on non-adjointable semi-Fredholm operators }
From \cite[Definition 3]{IM}  it follows as in the proof of \cite[Lemma 2.7.10]{MT} that $F$ has the matrix
$
\begin{pmatrix}
F_1 & 0  \\
0 & \tilde{F_4}  \\
\end{pmatrix}
$
w.r.t. the decomposition $U(M_{1}) \tilde{\oplus} U(N_{1}) ‎\stackrel{F}{\longrightarrow} M_{2} \tilde{\oplus} N_{2}.$ Obviously,  such operators are invertible in ${{B(l_{2}(\mathcal{A}))}_{/} }_{K(l_{2}(\mathcal{A}))} .$ Now, if only $N_{1}  $ is finitely generated, we say that $F$ has upper inper (Noether) decomposition, whereas if only $N_{2}  $ is finitely generated, we say that $F$ has lower unner (Noteher) decomposition. Based on \cite[Definition 4]{IM} we give now the following definition. \\
\begin{definition} \label{D02}  
	We say that $F$ has upper external (Noether) decomposition if there exist closed $C^{*}$-modules $ X_{1},X_{2}$ where $ X_{2} $ finitely generated, s.t. the operator $F_{0}$ defined as
	$$F_{0}=\begin{pmatrix}
	F & E_{2}  \\
	E_{3} & 0  \\
	\end{pmatrix}=l_{2}^{\prime}(\mathcal{A}) \oplus X_{1}  \longrightarrow l_{2}^{\prime \prime}(\mathcal{A}) \oplus X_{1}$$ 
	is invertible and s.t. $Im E_{2}  $ is complementable in $l_{2}^{\prime \prime}(\mathcal{A}). $ Similarly, we say that $F$ has lower external (Noether) decomposition if the above decomposition exists, only in this case we assume that $X_{1}$ is finitely generated and that $ \ker E_{3}  $ is complementable in $l_{2}^{\prime}(\mathcal{A}).$ 	
\end{definition}
\begin{proposition} \label{P06}  
	 $A$ bounded $\mathcal{A}$-operator $F=l_{2}^{\prime}(\mathcal{A}) \longrightarrow  l_{2}^{\prime \prime}(\mathcal{A})  $ admits an upper external (Noether) decomposition iff it admits an upper inner (Noether) decomposition. Similarly, $F$ admits a lower external (Noether) decomposition iff $F$ admils a lower inner (Noether) decomposition.
\end{proposition}
\begin{proof}
	As in the proof of \cite[Theorem 3]{IM}, we may let, when $F$ has an inner decomposition, the operator $F_{0}$ to be defined as 
	$$F_{0}=\begin{pmatrix}
	F_1 & F_2 & 0  \\
	0 & F_4 & id \\
	0 & id & 0 \\
	\end{pmatrix} : M_{1} \oplus N_{1} \oplus N_{2} \longrightarrow M_{2} \oplus N_{2} \oplus N_{1}.$$  Then $F_{0}  $ is invertible. Moreover, the operator $E_{2}: X_{1}=N_{2} \longrightarrow l_{2}^{\prime \prime}(\mathcal{A})=M_{2} \oplus N_{2}  $ is just the inclusion, hence  $Im E_{2}=N_{2} $ is complementable in $M_{2} \oplus N_{2}  = l_{2}^{\prime \prime}(\mathcal{A}) .$ Also, the operator $E_{3}: l_{2}^{\prime}(\mathcal{A}) =M_{1} \oplus N_{1} \longrightarrow X_{2} =N_{1}$ is simply the projection onto $ N_{1} $ along $M_{1}  ,$ so $ \ker E_{3} =M_{1}$ is complementable in $l_{2}^{\prime}(\mathcal{A}).$ To prove the other direction, when $F$ has an external decomposition, we may proceed in exactly the same way as in the proof of \cite[Theorem 3]{IM} . Indeed, to obtain (29) and (34), we use the assumptions in the definition of external decomposition that $Im E_{2}  $ and $\ker E_{3}$ are complementable in $l_{2}^{\prime \prime}(\mathcal{A})  $ and $l_{2}^{\prime}(\mathcal{A})  $ respectively.
\end{proof}
Clearly, any upper semi-Fredholm operator in the sense of our definition is also left invertible in ${{B(l_{2}(\mathcal{A}))}_{/} }_{K(l_{2}(\mathcal{A}))} ,$ whereas any lower semi-Fredholm operator is right invertible ${{B(l_{2}(\mathcal{A}))}_{/} }_{K(l_{2}(\mathcal{A}))}  $ (by upper and lower semi-Fredholm we mean here that $F$ admits upper and lower inner decomposition resp.). The converse also holds: 
\begin{proposition}
	If $F$ is left invertible in ${{B(l_{2}(\mathcal{A}))}_{/} }_{K(l_{2}(\mathcal{A}))}   ,$ then $F$ admitis upper inner decomposition. If $F$ is right invertible in ${{B(l_{2}(\mathcal{A}))}_{/} }_{K(l_{2}(\mathcal{A}))}   ,$ then it admitis lower inner decomposition. 
\end{proposition}
\begin{proof}
	If $GF=id+K^{\prime \prime}  $ for some $G:l_{2}^{\prime \prime}(\mathcal{A}) \longrightarrow l_{2}^{\prime }(\mathcal{A}) ,K^{\prime \prime} \in K(l_{2}(\mathcal{A})) ,$ then by following the proof of \cite[Theorem 5]{IM} we reach to (45) in \cite{IM}. Moreover, by this part of the proof of \cite[Theorem 5]{IM}, we also obtain that $G$ has the matrix
	$\begin{pmatrix}
	G_1 & G_2  \\
	0 & G_4  \\
	\end{pmatrix}  $ 
	w.r.t. the decomposition $l_{2}^{\prime \prime}(\mathcal{A})=M_{3} \oplus N_{3}  ‎\stackrel{G}{\longrightarrow} M_{2} \oplus N_{2}=l_{2}^{\prime}(\mathcal{A})$ where $G_{1}  $ is an isomorphism. Indeed, by (45) in \cite{IM} $ M_{3}=Im P =Im F K_{1}^{-1}p_{2}G .$ It follows that $M_{3}=F(M_{1}).$ Since $GF_{\mid_{M_{1}}}  $ is an isomorphism onto $M_{2} ,$ it follows that $G_{\mid_{F(M_{1})}}  $ is an isomorphism onto $M_{2}  .$ Then, considering the operator $G$ and applying the argumnets above, one deduces the second statement in the proposition. 
\end{proof}
The next lemma is again a corollary of \cite[Theorem 5]{IM}:
\begin{lemma} \label{L02}  
	Let $F,G$ be bounded $\mathcal{A}$-operators and suppose that $GF$ is Fredholm. Then there exist decompositions 
		$$l_{2}^{\prime}(\mathcal{A})=M_{1} \oplus N_{1}  ‎\stackrel{F}{\longrightarrow} l_{2}^{\prime \prime}(\mathcal{A})=M_{3} \oplus N_{3} \stackrel{G}{\longrightarrow} l_{2}^{\prime}(\mathcal{A})=M_2\oplus N_{2}$$
	w.r.t. which $F,G$ have matrices
	 $\begin{pmatrix}
	 F_1 & 0  \\
	 0 & F_4  \\
	 \end{pmatrix} ,$
	 $\begin{pmatrix}
	 G_1 & G_2  \\
	 0 & G_4  \\
	 \end{pmatrix} ,$ 
	respectively, where $F_{1} , G_{1} $ are isomorphisms, $N_{1},N_{2}$ are finitely generated.
\end{lemma} 
From now on, throughout this section we will let ${{\mathcal{M}\Phi}}_{+}(l_{2}(\mathcal{A}))$ denote the set of all operators left invertible in $B(l_{2}(\mathcal{A})) / K(l_{2}(\mathcal{A})),$ whereas ${{\mathcal{M}\Phi}}_{-}(l_{2}(\mathcal{A}))$ will denote the set of all operators right invertible in $B(l_{2}(\mathcal{A})) / K(l_{2}(\mathcal{A})).$ Then we set  ${{\mathcal{M}\Phi}}(l_{2}(\mathcal{A})) = {{\mathcal{M}\Phi}}_{+}(l_{2}(\mathcal{A})) \cap {{\mathcal{M}\Phi}}_{-}(l_{2}(\mathcal{A}))$ Although the notation here coincides with notation in \cite{I} we do not assume the adjointability of operators here in this section.\\
Most of the results from \cite{I}, \cite{S} are also valid when we consider the non-adjointable semi-Fredholm operators and the same proofs can be applied. Here we are going slightly differnt fomulations and proofs of some of the results from   \cite{I}, \cite{S} which can not be transfered directly to the non-adjointable case.
\begin{lemma} \label{L07}  
	Let V be a finitely generated Hilbert submodule of $l_{2}(\mathcal{A}),$ $F \in B(l_{2}(\mathcal{A})$ and suppose that 
	$P_{V^{\bot}} F \in {{\mathcal{M}\Phi}}(l_{2}(\mathcal{A})),V^{\perp} (l_{2}(\mathcal{A}),V^{\bot})$ where $P_{V^{\bot}}$ is the orthogonal projection onto $V^{\bot}$ along V. Then $F \in {{\mathcal{M}\Phi}}_{-}(l_{2}(\mathcal{A})) .$
\end{lemma}
\begin{proof}
	Since V is finitely generated, by \cite[Lemma 2.3.7]{MT}, V is an orthogonal direct summand in $l_{2}(\mathcal{A})$, so $l_{2}(\mathcal{A})=V  \oplus V^{\bot}$. Consider the decomposition
	$$l_{2}(\mathcal{A}) = M_{1} \tilde \oplus {N_{1}}_{\overrightarrow{P_{V^{\bot}} F}} M_{2} \tilde \oplus N_{2}= V^{\bot} $$
	w.r.t. which $P_{V^{\bot}} F$ has the matrix
	\begin{center}
		$	\left\lbrack
		\begin{array}{ll}
		(P_{V^{\bot} }F)_{1} & 0 \\
		0 & P_{V^{\bot}} F)_{4} \\
		\end{array}
		\right \rbrack
		$
	\end{center}
	where $N_{1},N_{2}$ are finitely generated and $(P_{V^{\bot} }F)_{1}$ is isomorphism. Since $(P_{V^{\bot} }F)_{1}=P_{M_{2}}^{V^{\bot}}P_{V^{\bot}} F_{{\mid}_{{M}_{1}}}$ 
	where $P_{M_{2}}^{V^{\bot}}$ is the projection of $V^{\bot}$ onto $M_{2}$ along $N_{2},$ it follows that $P_{M_{2}}^{V^{\bot}}P_{V^{\bot}} F_{{\mid}_{{M}_{1}}}$ , is an isomorphism of $M_{1}$ onto $M_{2}$. But $l_{2}(\mathcal{A})=M_{2}\tilde \oplus N_{2}\tilde \oplus V$ and $P_{M_{2}}^{V^{\bot}}P_{V^{\bot}}=P_{M_{2}}$ where $P_{M_{2}}$ is the projection of $l_{2}(\mathcal{A})$ onto $M_{2}$ along $N_{2}\tilde \oplus V$. Hence 
	F has the matrix
	\begin{center}
		$\left\lbrack
		\begin{array}{ll}
		F_{1} & F_{2} \\
		F_{3} & F_{4} \\
		\end{array}
		\right \rbrack
		$
	\end{center}
	w.r.t. the decomposition $$l_{2}(\mathcal{A}) = M_{1} \tilde \oplus {N_{1}}_{\overrightarrow{F}}  M_{2} \tilde \oplus (N_{2} \tilde \oplus V)= l_{2}(\mathcal{A}) $$
	where $F_{1}= P_{M_{2}}F_{{\mid}_{{M}_{1}}}$ an isomorphism. Then w.r.t. the decomposition 
	$$l_{2}(\mathcal{A})=U_{1}( M_{1}) \tilde \oplus U_{1}(N_{1} )_{\overrightarrow{F}}  U_{2}^{-1}( M_{2}) \tilde \oplus U_{2}^{-1}(N_{2}\tilde \oplus V )=  l_{2}(\mathcal{A}) $$
	F has the matrix
	\begin{center}
		$\left\lbrack
		\begin{array}{ll}
		\overline{F}_{1} & 0 \\
		0 & \overline{F}_{4}  \\
		\end{array}
		\right \rbrack
		$
	\end{center}
	where
	\begin{center}
		$U_{1}=\left\lbrack
		\begin{array}{ll}
		1 & -F_{1}^{-1}F_{2} \\
		0 & 1 \\
		\end{array}
		\right \rbrack
		,$
	\end{center} 
	\begin{center}
		$U_{2}=\left\lbrack
		\begin{array}{ll}
		1 & 0 \\
		-F_{3}F_{1}^{-1} & 1 \\
		\end{array}
		\right \rbrack
		,$
	\end{center} 
	and $\tilde{F_{1}} $ are isomorphisms. Now, $N_{2}\tilde \oplus V$ is finitely generated, hence $U_{2}^{-1}(N_{2}\tilde \oplus V)$ is finitely generated also. 
\end{proof}
\begin{lemma} \label{L03}  
	Let $G,F \in B(l_{2}(\mathcal{A}))  ,$ suppose that $Im G$ is closed and that $\ker G$ and $Im G$ are complementable in $l_{2}(\mathcal{A}).$ If $GF \in  {\mathcal{M}}\Phi _{-}(l_{2}(\mathcal{A})) $ then $\sqcap  F \in  {\mathcal{M}}\Phi _{-}(l_{2}(\mathcal{A})), N  $ where $\ker G \tilde{\oplus} N= l_{2}(\mathcal{A}) $  and $\sqcap  $ denotes the projection onto $N$ along $\ker G.$ 
\end{lemma}
\begin{proof}
	By the arguments from the proof of Lemma \ref{L02}, since \\
	$GF \in {\mathcal{M}}\Phi _{-}(l_{2}(\mathcal{A})),$ there exists a chain of decompositions 
	$$ l_{2}(\mathcal{A}) = M_{1} \tilde \oplus {M_{2}}‎‎\stackrel{F}{\longrightarrow} {R_{1} \tilde \oplus R_{2}} \stackrel{G}{\longrightarrow} {N_{1} \tilde \oplus N_{2}} $$ 
	w.r.t. which $F$ and $G$ have matrices 
	$\begin{pmatrix}
		F_1 & 0  \\
		0 & F_4  \\
	\end{pmatrix}
	$,
	$\begin{pmatrix}
	G_1 & G_{2}  \\
	0 & G_4  \\
	\end{pmatrix}
	$
	wher $F_{1},G_{1}$ are isomorphisms and $N_{2}$ is finitely generated Indeed, considering the $\mathcal{M}\Phi _{-} $ decomposition $M_{1} \tilde \oplus {M_{2}}‎‎\stackrel{GF}{\longrightarrow} {N_{1} \tilde \oplus N_{2}}   ,$ the arguments of the proof of  until (45) in \cite{MT} applies also in the case when $N_{1}  $ on $N_{2}  $ are not finitely generated. Hence $G$ has the matrix
	$\begin{pmatrix}
	G_1 & 0  \\
	0 & \tilde{G_4}  \\
	\end{pmatrix}
	$
	w.r.t. the decomposition $R_{1} \tilde \oplus {U(R_{2})}‎‎\stackrel{G}{\longrightarrow} {N_{1} \tilde \oplus N_{2}} $ where $U$ is an isomorphism. It is not hard to see that $ker G  \subseteq U(R_{2})  .$ Since $\ker G \tilde{\oplus} N= l_{2}(\mathcal{A}) $ and $ker G  \subseteq U(R_{2})  ,$ we get that $U(R_{2}) = \ker G \tilde{\oplus} (U(R_{2}) \cap N).$ As $Im G$ is closed, $G_{\mid_{N}}  $ is an isomorphism onto $Im G$ by open mapping theorem. Hence $G_{\mid_{(U(R_{2}) \cap N)}}  $ is an isomorphism. Thus $Im G=N_{1} \tilde{\oplus} G (U(R_{2}) \cap N)  .$ As $Im G$ is complementable in $l_{2}(\mathcal{A})  ,$ we have that $ G (U(R_{2}) \cap N)  $ is also complementable in $ l_{2}(\mathcal{A}) .$ Since $G (U(R_{2}) \cap N) \subseteq N_{2} ,$ it follows that $G (U(R_{2}) \cap N)  $ is complementable in $N_{2}  $ also. But $N_{2}  $ is finitely generated, hence $G (U(R_{2}) \cap N)  $ must be finitely generated being a direct summand in $N_{2}   .$ Hence $U(R_{2}) \cap N  $ is finitely generated being isomorphic to $G (U(R_{2}) \cap N)  .$ W.r.t. the decomposition $M_{1} \tilde \oplus {M_{2}}‎‎\stackrel{F}{\longrightarrow} R_{1} \tilde \oplus U(R_{2})  , $ $F$ has the matrix 
	$
	\begin{pmatrix}
	F_1 & \tilde{F_{2}}  \\
	0 & \tilde{F_{4}}  \\
	\end{pmatrix},
	$
	hence $F$ has the matrix
	$\begin{pmatrix}
	F_1 & 0  \\
	0 & \tilde{\tilde{F_{4}}}  \\
	\end{pmatrix},  $ 
	w.r.t. the decomposition $M_{1} \tilde \oplus {\tilde{U}(M_{2})}‎‎\stackrel{F}{\longrightarrow} R_{1} \tilde{\oplus} U(R_{2}) $ where $\tilde{U}$ is an isomorphism. Moreover, since $l_{2}(\mathcal{A})= R_{1} \tilde{\oplus} (U(R_{1}) \cap N) \tilde{\oplus} \ker G  ,$ it follows that $\sqcap_{\mid_{R_{1}}}  $ is an isomorphism (recall that $ \sqcap $ is the projection onto $N$ along $\ker G  .$) It is then easy to see that $\sqcap F  $ has the matrix
	$\begin{pmatrix}
	(\sqcap F)_1 & 0  \\
	0 &  (\sqcap F)_4  \\
	\end{pmatrix},  $ 
	w.r.t. the decomposition $M_{1} \tilde {\oplus} {\tilde{U}(M_{2})}‎‎\stackrel{\sqcap F}{\longrightarrow} \sqcap (R_{1}) \tilde{\oplus} (U(R_{1}) \cap N)  $ where $(\sqcap F)_{1}  $ is an isomorphism. Now, $ U(R_{1}) \cap N $ is finitely generated.
\end{proof}
Recall now the definition of classes ${{\mathcal{M}}\Phi} _{+}^{- \prime}(l_{2}(\mathcal{A})),{{\mathcal{M}}\Phi} _{+}^{+ \prime}(l_{2}(\mathcal{A}))  ,$ from \cite{I}. Again we are going to use the same notation, but we are not going to assume adjointability.
\begin{lemma} \label{L04}  
	$ F \in B((l_{2}(\mathcal{A}))) $ admits upper external (Noether) decomposition with the property that $X_{2} \preceq X_{1}  $ iff $F \in {{\mathcal{M}}\Phi} _{+}^{- \prime}(l_{2}(\mathcal{A}))  .$ Similarly $F$  admits lower external (Noether) decomposition with the property that $X_{1} \preceq X_{2}  $  iff $F \in {{\mathcal{M}}\Phi} _{+}^{- \prime}(l_{2}(\mathcal{A}))    .$ 
\end{lemma} 
\begin{proof}
	 Statements can Be shown in a similar way as in the proof of Proportion \ref{P06}. 
\end{proof}
\begin{lemma} \label{L05}  
	 Let $F \in {{\mathcal{M}}\Phi} _{-}^{+ \prime}(l_{2}(\mathcal{A}))    .$ Then $ F+K \in {{\mathcal{M}}\Phi} _{-}^{+ \prime}(l_{2}(\mathcal{A}))  $ for all $K \in K(l_{2}(\mathcal{A})).$ 
\end{lemma}
\begin{proof}
	Let $l_{2}(\mathcal{A})= M_{1} \tilde \oplus {M_{2}}‎‎\stackrel{F}{\longrightarrow} N_{1} \tilde \oplus N_{2}=l_{2}(\mathcal{A})  $ be an ${{\mathcal{M}}\Phi} _{-}^{+ \prime}  $ decomposition for $F.$ Then $N_{2}  $ is finitely generated and $N_{2} \preceq  N_{1}  .$ We may assume that $N_{2} \preceq L_{n}, L_{n}=N_{2} \tilde{\oplus} P   $ and $M_{2}= L_{n}^{\perp} \oplus P $ for some $n \in \mathbb{N}  $ and $P$ fintely generated. Moreover, we may cloose an $ n $ big enough s.t. $ 	\parallel q_{n}K	\parallel < \parallel F_{1}^{-1} \parallel^{-1} .$ Then we may proceed as in the proof of \cite[Lemma 2.7.13]{MT} to and use that $ N_{2} \preceq N_{1}  $ to deduce the lemma.	
\end{proof}

As regards \cite{S}, we need to slight reformulate some definitions and results from that paper when we consider the nonadjointable case. 
\begin{definition} \label{D03}  
	We set $\widehat{\widehat{{\mathcal{M}\Phi}}}_{+}^{-}(l_{2}(\mathcal{A}))  $ to be as the set $\widehat{{\mathcal{M}\Phi}}_{+}^{-}( H_{\mathcal{A}})  $ in \cite{I2}, but we demand that $ R(PF_{\mid_{ R(P) }}) $ should be complementable in $R(P)  ,$ instead of the adjointability of $P.$ 
\end{definition}
Recall from \cite{S} that $P(l_{2}( \mathcal{A} ))  $ denote the set of projections, not necessarily adjointable, with finitely generated kernel. Put 
$$\sigma_{e \tilde{a} 0}^{ \mathcal{A} } (F)= \lbrace \alpha \in Z(  \mathcal{A}  ) \mid (F- \alpha I) \notin  \widehat{\widehat{{\mathcal{M}\Phi}}}_{+}^{-}(l_{2}(\mathcal{A})) \rbrace.$$ 
Then we have the following non adjointability version of \cite[Theorem 2]{I2} :
\begin{theorem} \label{T01}  
	For $F \in B(l_{2}(\mathcal{A}))  $ we have 
	$$\sigma_{e \tilde{a} 0}^{ \mathcal{A} } (F)=\cap \lbrace \sigma_{a 0}^{ \mathcal{A} } (PF_{\mid_{ R(P) }}) \mid P \in P (l_{2}(\mathcal{A}))  \rbrace  $$
	where $\sigma_{a 0}^{ \mathcal{A} } (PF_{\mid_{ R(P) }})= \lbrace \alpha \in Z(  \mathcal{A}  ) \mid (PF- \alpha I)_{R(P)} \rbrace $ is bounded below on $R(P)$ or that $ R(PF - \alpha P) $ is complementable in $R(P) \rbrace .$ 
\end{theorem}
\begin{proof}
	If $ \alpha \notin  \sigma_{a 0}^{ \mathcal{A} } (PF_{\mid_{ R(P) }})$ for some $P \in P (l_{2}(\mathcal{A}))   ,$ then $(PF-\alpha I)_{\mid_{R(P)}}  $ is bounded below and $R(PFP-\alpha P) $ is complementable in $R(P).$ Hence we may proceed as in the proof of the \cite[Theorem 10]{S}, to deduce that $F-\alpha I \in \widehat{\widehat{{\mathcal{M}\Phi}}}_{+}^{-}(l_{2}(\mathcal{A}))   .$ Conversely, if $\alpha \in Z(\mathcal{A}) \setminus { \sigma_{e \tilde{a} 0}^{ \mathcal{A} } (F)},$ then by the proof of \cite[Theorem 10]{S} we obtain a decomposition 
	$$l_{2}(\mathcal{A})=V^{-1}(M_{2}) \tilde{\oplus} N_{2} =V^{-1}(M_{2}) \tilde{\oplus} N_{2}^{\prime \prime }  \tilde{\oplus} N_{2}^{\prime }=N  \tilde{\oplus} N_{2}^{\prime }   $$ 
	and $N_{2}^{\prime } \cong N_{1},N_{2}= N_{2}^{\prime } \tilde{\oplus} N_{2}^{\prime \prime },$ $U,V$ are isomorphism, $N_{1} $ is finitely generated and $(F-\alpha I)_{{\mid}_{N}}  $ maps $N$ isomorphiscally onto $V^{-1}(M_{2})  .$ 	If we let, as in that proof, $P$ be the projection ont $N$ along $N_{2}^{\prime }  ,$ then $P_{\mid_{V^{-1}(M_{2}) \tilde{\oplus} N_{2}^{\prime  }}} $ is an isomorphism onto $N.$ Set $\tilde{N}=P(V^{-1}(M_{2})), \tilde{\tilde{N}}=P(N_{2}^{\prime })  .$ We have then $N=\tilde{N}=\tilde{\tilde{N}}  .$ Hence $P(F-\alpha I)_{{\mid}_{N}}  $ is an isomorphism onto $\tilde{N}  $ which is complementable in $N=R(P)  ,$ so $\alpha \notin  \sigma_{a 0}^{ \mathcal{A} } (PF_{\mid_{ R(P) }})  .$
\end{proof}
Remark: It can be shown that $\widehat{\widehat{{\mathcal{M}\Phi}}}_{+}^{-}(l_{2}(\mathcal{A})) $ is open. \\
\\
Set now $\widehat{\widehat{{\mathcal{M}\Phi}}}_{-}^{+}(l_{2}(\mathcal{A}))   $ to be the set as $\widehat{{\mathcal{M}\Phi}}_{-}^{+}( H_{\mathcal{A}}) $ in \cite{S}, only we do not demand the adjointability of the projection $P$ onto $M_{1}^{\prime}  \tilde{\oplus} N $ along $N_{2}^{\prime}  ,$ but we require that $R(P)$ splits into $R(P)=\tilde{\tilde{N}} \tilde{\oplus} \tilde{N}$ s.t. $PG_{\mid_{\tilde{N}}}$  is an isomorphism from $\tilde{N}  $ onto $R(P).$ Then we put 
$$\sigma_{e \tilde{d} 0}^{ \mathcal{A} } (G)= \lbrace \alpha \in Z(  \mathcal{A}  ) \mid (G- \alpha I) \notin \widehat{\widehat{{\mathcal{M}\Phi}}}_ {+}^{-}(l_{2}(\mathcal{A})) \rbrace $$
and reach to the following non adjointable analogue of \cite[Theorem 11]{S}. 
\begin{theorem} \label{T02}  
	For $G \in B(l_{2}(\mathcal{A})) \rbrace $ we have $$\sigma_{e \tilde{d} 0}^{ \mathcal{A} } (G)=\cap \lbrace \sigma_{d 0}^{ \mathcal{A} } (PG_{\mid_{ R(P) }}) \mid P \in P (l_{2}(\mathcal{A})) \rbrace   $$ where $\sigma_{d 0}^{ \mathcal{A} } (PG_{\mid_{ R(P) }})= \lbrace \alpha \in Z(  \mathcal{A}  ) \mid {R(P)} \rbrace  $ does not split into the decomposition $R(P)=\tilde{N}  \tilde{\oplus} \tilde{\tilde{N}}  $ where $PG_{\mid_{\tilde{N}}}  $ is an isomorphism onto $R(P)\rbrace .$ 
\end{theorem}

\begin{proof}
	If $\alpha \notin  \sigma_{d 0}^{ \mathcal{A} } (PG_{\mid_{ R(P) }})  $ for some $ P \in P (l_{2}(\mathcal{A}))  ,$ then $R(P)=\tilde{N}  \tilde{\oplus} \tilde{\tilde{N}}  $ for some closd submodules $\tilde{N}, \tilde{\tilde{N}}  $ or $ R(P) $ s.t. $(PG-\alpha I)  $ is an isomorphism onto $R(P).$
	Letting $\tilde{\tilde{N}}  $ play the role of $N(PD-\alpha I)  $ in the proof of \cite[Theorem 11]{S}, we may proceed in the same way as in that proof to conclude that $G-\alpha I \in \widehat{\widehat{{\mathcal{M}\Phi}}}_{-}^{+}(l_{2}(\mathcal{A}))  .$ On the other hand, if $\alpha \in Z(\mathcal{A}) \setminus \sigma_{e \tilde{d} 0}^{ \mathcal{A} } (G) ,$ then $ G-\alpha I \in \widehat{\widehat{{\mathcal{M}\Phi}}}_{-}^{+}(l_{2}(\mathcal{A}))  .$ As in the proof of \cite[Theorem 11]{S} (and using the same notation) we may consider the projection $P$ onto $M_{1}^{\prime} \tilde{\oplus} N $ along $N_{2}^{\prime}  $ and obtain that $P(G- \alpha I)_{\mid_{M_{1}^{\prime}}}  $ is an isomorphism onto $M_{1}^{\prime}  \tilde{\oplus} V .$ 
\end{proof}
\begin{remark}
	Similarly as for $\widehat{\widehat{{\mathcal{M}\Phi}}}_{+}^{-}(l_{2}(\mathcal{A}))  ,$ one can show that $\widehat{\widehat{{\mathcal{M}\Phi}}}_{-}^{+}(l_{2}(\mathcal{A}))  $ is open. 
\end{remark}
\section{On semi-$\mathcal{A}$-$B$-Fredholm operators  }
\begin{lemma} \label{L08}  
	Let $F \in B^{a}(M) $ where $M$ is a Hilbert $C^{*}$-module and suppose that $Im F$ is closed. Then\\
	a) $F \in \mathcal{M}\Phi_{+}(M)   ,$ iff $\ker F$ is finitely generated.\\
	b) $F \in \mathcal{M}\Phi_{-}(M)  ,$ iff $Im F^{\perp}  $ is finitely generated.
\end{lemma}
\begin{proof}
	a) Let $M=M_{1} \tilde \oplus M_{2}‎‎\stackrel{F}{\longrightarrow} M_{1}^{\perp} \tilde \oplus M_{2}^{\perp}=M   $ be an $\mathcal{M}\Phi_{+}  $ decomposition for $F.$ By the arguments from the proof of \cite[Proposition 3.6.8]{MT}, it is not hard to see that $\ker F \subseteq M_{2}  .$ Now, by \cite[Theorem 2.3.3]{MT}, $ \ker F$ is orthogonally complementable in $M,$ hence in $ M_{2} ,$ as $ \ker F \subseteq M_{2} .$ Since $M_{2}  $ is finitely, it follows that $\ker F$ is finitely generated, being a direct summand in $M_{2}.$ 
	Conversely, if $\ker F$ is finitely generated, then 
	$$ H_{\mathcal{A}}=\ker F^{\perp} \oplus \ker F \stackrel{F}{\longrightarrow} Im F \oplus Im F^{\perp}=H_{\mathcal{A}}$$
	is an $\mathcal{M}\Phi_{+}$ dcomposition for F. (Here we use that $Im F$ is closed.).\\
	b) This can be shown by passing to the adjoints and using a). Use that $Im F^{*}$ is closed if and only if $Im F$ is closed by the proof of \cite[Theorem 2.3.3]{MT} part ii). Moreover, $F \in \mathcal{M}\Phi_{-}(M) $ iff $ F^{*} \in \mathcal{M}\Phi_{+}(M) $ by \cite[Corollery 2.11]{I} and $Im F^{\perp}=\ker F^{*}.$
\end{proof}
\begin{definition} \label{D04}  
	Let $F \in B^{a}(H_{\mathcal{A}})  .$ Then $F$ is said to be upper semi-$\mathcal{A}$-$B$-Fredhom if the following holds:
	1)	$Im F^{m} $ is closed for all $ m $
	2)	There exists an  $n$ s.t. $F_{\mid_{Im F^{n}}}  $ upper semi -$\mathcal{A}$- Fredholm. 
\end{definition}
Similarly, $F$ is said to be lower semi-$\mathcal{A}$-$B$-Fredholm if 1) and 2) hold, only in this case we assume in 2) that $F_{\mid_{Im F^{n}}}  $ is lower semi-Fredholm. Finally, if $F_{\mid_{Im F^{n}}} $ is $\mathcal{A}$-Fredholm, we say that $F$ is $\mathcal{A}$-$B$-Fredholm.
\begin{proposition}  \label{P04}  
	If $F$ is upper semi-$\mathcal{A}$-$B$-Fredholm (respectively lower semi-$\mathcal{A}$-$B$-Fredholm), then $F_{\mid_{Im F^{m}}}  $ is upper semi-$\mathcal{A}$-Fredholm (respectively lower semi-$\mathcal{A}$-Fredholm) for all $ m \geq n .$ Moreover, if $F$ is $\mathcal{A}$-$B$-Fredholm and $Im F^{n}\cong H_{\mathcal{A}}  ,$ then $Im F^{m}\cong H_{\mathcal{A}} ,$ for all $ m \geq n ,$ $F_{\mid_{Im F^{m}}}  $ is $\mathcal{A}  $ Fredholm for all $ m \geq n $ and index $F_{\mid_{Im F^{m}}}= \index{F_{\mid_{Im F^{n}}}}  $ for all $ m \geq n .$
\end{proposition}
\begin{proof}
	We will prove this by induction. Since $ Im F^{n+1}=Im F_{\mid_{Im F^{n}}} $ and $Im F^{n+1}  $ is closed by assumption, by \cite[Theorem 2.3.3]{MT} applied to the operator $F_{\mid_{Im F^{n}}}  ,$ we deduce that $ \ker F_{\mid_{Im F^{n}}} $ and $ Im F^{n+1}$ are orthogonally complementable in $ Im F^{n} .$ Nanely, by \cite[Theorem 2.3.3]{MT} applied to $ F^{n}  $ we have that $ Im F^{n}  $ is orthogonally complementable in $ H_{\mathcal{A}}  ,$ as $ Im F^{n}$ is closed.
	Hence $F_{\mid_{Im F^{n}}} \in B^{a}(Im F^{n} )   $ so we can indeed apply \cite[Theorem 2.3.3]{MT} on $F_{\mid_{Im F^{n}}}   .$ If $F_{\mid_{Im F^{n}}}  $ is upper semi-$\mathcal{A}$-Fredholm operator, by Lemma \ref{L08} we have that $ \ker F_{\mid_{Im F^{n}}} = \ker F \cap Im F^{n}  $ is finitely generated, as $Im F_{\mid_{Im F^{n}}} $ is closed. If $F_{\mid_{Im F^{n}}}  $ is lower semi-$\mathcal{A}$-Fredholm, then again  by Lemma \ref{L08}, if we let $R$ denote the orthogonal complement of $Im F^{n-1}  $ in $ Im F^{n}  $, we get that $R$ is finitely generated. Consider now the operator $F_{\mid_{  Im F^{n+1}  }}.$ Again, $ Im (F_{\mid_{  Im F^{n +1 }  }})= Im F^{n +2 }  $ is closed by assumption, so by the same arguments as above we may apply \cite[Theorem 2.3.3]{MT} on $ F_{\mid_{  Im F^{n+1  }  }}   $ to deduce that $\ker  F_{\mid_{  Im F^{n +1 }  }}=\ker F \cap Im F^{n +1 }    $ is orthogonally complementable in $ Im F^{n+1  } .$ Since $ Im F^{n+1  }  $ is orthogonally complementable in $H_{\mathcal{A}}  ,$ so is $\ker F \cap Im F^{n +1 }  $ as well. Now, since we have $ \ker F \cap Im F^{n +1 }  \cap  Im F^{n  } ,$ it follows that $\ker F \cap Im F^{n +1 } \oplus M= \ker F \cap  Im F^{n  }   ,$ where $ M=(\ker F \cap Im F^{n }) \cap ((\ker F \cap Im F^{n +1 })^{\perp})  .$ Since $\ker F \cap Im F^{n }  ,$ when $ F_{\mid_{  Im F^{n  }  }}  $ is upper semi-$\mathcal{A}$-Fredholm is finitely generated, it follows that $\ker F \cap Im F^{n +1 }  $ is finitely generated being a direct summand in $\ker F \cap Im F^{n }  .$ Thus  by Lemma \ref{L08} $ F_{\mid_{  Im F^{n +1 }  }}   $ is upper semi-$\mathcal{A}$-Fredholm, when $F_{\mid_{  Im F^{n  }  }}  $ is so. Next, again by the same arguments as earlier we get that $ Im F^{n+2 } \oplus X= Im F^{n + 1  }  $ for some closed submodule $X$ (using that $Im(F_{\mid_{  Im F^{n +1 }  }}) = Im F^{n+2 } $ is closed). By the proof of Proposition \ref{P01}, replacing by $F$ and $D$ by $F_{\mid_{  Im F^{n  }  }}  $ we obtain that $R \cong S(X) \tilde{\oplus} M $ where $S$ is an isomorphism. (recall that $ Im F^{n +1 } \oplus R =  Im F^{n  } )$ If $F_{\mid_{  Im F^{n  }  }}   $ is lower semi-$\mathcal{A}$-Fredholm, then $R$ is finitely generated, as we have seen. Hence $X$ must be finitely generated also. Thus $F_{\mid_{  Im F^{n  }  }}  $ is lower semi-$\mathcal{A}$-Fredholm in this case by Lemma \ref{L08}. Finally, if $ F_{\mid_{  Im F^{n  }  }}  $ is $\mathcal{A}$-Fredholm, then by Lemma \ref{L08} both $\ker F_{\mid_{  Im F^{n  }  }} = \ker F \cap  Im F^{n  }  $ and the orthogonal complement of $Im F^{n +1 }  $ in $ Im F^{n  } $ are finitely generated. Thus $ Im F^{n  }= Im F^{n+1  } \oplus R^{\prime} $ for some finitely generated closed submodule $R^{\prime}$. Hence, if $H_{\mathcal{A}} \cong Im F^{n  }  ,$ by Dupre-Filmore theorem $ Im F^{n +1 }  \cong  H_{\mathcal{A}} $ as well. By the same arguments as above we can deduce that both $ \ker F_{\mid_{  Im F^{n + 1 }  }}   $ and the orthogonal complement of $Im F^{n +2 }   $ in $Im F^{n +1 }  $ are finitely generated, as both $\ker F_{\mid_{  Im F^{n  }  }}    $ and $ R^{\prime} $ are so. Hence $F_{\mid_{  Im F^{n+1  }  }}   $ is $\mathcal{A}$-Fredholm and since $  Im F^{n +1 } \cong  H_{\mathcal{A}}  ,$ by \cite[Theorem 2.7.9]{MT} the index of $ F_{\mid_{  Im F^{n +1 }  }} $ is well-defined. If we let $X^{\prime}  $ denote the orthogonal complement of $Im F^{n +2 }   $ in $ Im F^{n+1  }  $ and $M^{\prime}  $ denote the orthogonal complement of $\ker F \cap Im F^{n+1  }    $ in $\ker F \cap Im F^{n  }  ,$ by the same arguments as earlier we get that $ R^{\prime} \cong X^{\prime} \oplus M^{\prime}.$ Hence we get  $ index{ F_{\mid_{  Im F^{n +1 }  }} } = [\ker F \cap Im F^{n+1}]-[X^{\prime}]=[\ker F \cap Im F^{n +1 } ] +[M^{\prime}]-[X^{\prime}]-[M^{\prime}] =[\ker F \cap Im F^{n}] - [R^{\prime}]= index{ F_{\mid_{  Im F^{n }  }} }.$
\end{proof}
For an $\mathcal{A}$-$B$-Fredholm operator $F$, we set $index F= index{ F_{\mid_{  Im F^{n }  }} },$ where $n$ is as in the Definition \ref{D04} above. 
\begin{lemma} \label{L06}  
	Let $F \in  {\mathcal{M}\Phi}(H_{\mathcal{A}})  ,$ let $P \in B (H_{\mathcal{A}}) $ s.t. $P$ is the projection and $N(P) $ is finitely generated. Then $PF_{\mid_{R(P)}} \in {\mathcal{M}\Phi} (R(P)) $ and $index PF_{\mid_{R(P)}}=index F .  $
\end{lemma}
\begin{proof}
	From \cite[Lemma 1]{S}, we already know that $PF_{\mid_{R(P)}} \in {\mathcal{M}\Phi} (R(P)).$ If remains to show that $index PF_{\mid_{R(P)}}=index F  .$ Now, since $ P \in  {\mathcal{M}\Phi}(H_{\mathcal{A}})  ,$ by \cite[Lemma 2.7.11]{MT}, $index PFP=index P+ index F+ index P= index F,$ as  $index P=0  .$ By the proof of \cite[Lemma 1]{S}, there „exists decompositia $R(P)=P(M) \oplus \tilde{N} \stackrel{PF}{\longrightarrow} M^{\prime} \oplus \tilde{N}^{\prime}=R(P)  $ w.r.t. which $PF$ has the matrix 
	$\left\lbrack
	\begin{array}{ll}
	(PF)_{1} & (PF)_{2} \\
	0 & (PF)_{4} \\
	\end{array}
	\right \rbrack,
	$  
	where $(PF)_{1}  $ is an isomorphism, $\tilde{N}, \tilde{N}^{\prime}  $ are finitely generated. In addition $ P $ has the matrix
	$\left\lbrack
	\begin{array}{ll}
	P_{1} & P_{2} \\
	0 & P_{4} \\
	\end{array}
	\right \rbrack,
	$  
	w.r.t. the decomposition
	$$H_{\mathcal{A}}=M \tilde{\oplus} N \longrightarrow P(M) \tilde{\oplus} (\tilde{N} \oplus N(P))=H_{\mathcal{A}}  $$ 
	where $ P_{1} $ is an isomorphism and $N$ is finitely generated. Moreover, 
	$$H_{\mathcal{A}}=M \tilde{\oplus} N \stackrel{PFP}{\longrightarrow} M^{\prime} \tilde{\oplus} N^{\prime}= H_{\mathcal{A}} $$ 
	is an ${\mathcal{M}\Phi}_{-}$ decomposition for $PFP$ and $N^{\prime} \cong \tilde{N}^{\prime} \oplus N(P)  .$ Since $index PFP= index F  ,$ it follows that $[N]-[N^{\prime}] =index F $ in $K(\mathcal{A}).$ Next, it is easily seen, by diagonalizing the matrix
	$\left\lbrack
	\begin{array}{ll}
	P_{1} & P_{2} \\
	0 & P_{4} \\
	\end{array}
	\right \rbrack,
	$  
	as in the proof of \cite[Lemma 2.7.10]{MT} that $[N]-[\tilde{N}]-[N(P)]=[N]-[\tilde{N} \oplus N(P) ]= index P = 0  .$ Similarly, by diagonalizing the matrix
	$\left\lbrack
	\begin{array}{ll}
	(PF)_{1} & (PF)_{2} \\
	0 & (PF)_{4} \\
	\end{array}
	\right \rbrack,
	$  
	we obtain that $index (PF_{\mid_{R(P)}})=[\tilde{N}]-[\tilde{N}^{\prime}] .$ Finaly, $ [\tilde{N}^{\prime}] + [N(P)]= [N^{\prime}]  .$ Combining all this together, we obtain $index (PF_{\mid_{R(P)}})= [\tilde{N}]-  [\tilde{N}^{\prime}]= [\tilde{N}]+ [N(P)]-  [\tilde{N}^{\prime}]- [N(P)]=[\tilde{N} \oplus N(P)] -[\tilde{N}^{\prime} \oplus N(P) ] = [N] - [N^{\prime}] = index F.$
\end{proof}
\begin{theorem} \label{T03} 
	Let $T$ be an $\mathcal{A}$-$B$-Fredholm operator on $H_{\mathcal{A}}  ,$ and suppose that mis such that $T_{\mid_{ Im T^{m}}}  $ is $\mathcal{A}$-Fredholm and $Im T^{n}  $ is closed for all $ n \geq m .$ Let $F$ be a finite rank operator (that is $Im F$ is finitely generated) and suppose that $ Im (T+F)^{n}  $ is closed for all $n \geq m   .$ Finally assume that $ Im T^{m} \cong H_{\mathcal{A}}  $ and that $Im (\tilde{F}),  T^{m} (\ker \tilde{F}),T^{m} (\ker \tilde{F}^{\perp}),(T+F)^{m} (\ker \tilde{F}^{\perp})  $ are closed, where $ \tilde{F}=(T+F)^{m}-T^{m} .$ Then $T+F$ is an $\mathcal{A}$-$B$-Fredholm operator and $index T+F=index T.$
\end{theorem}
\begin{proof}
	Observe first that since $\tilde{F} \in B^{a}(H_{\mathcal{A}})  $ and $Im \tilde{F}  $ is closed by assumption, we have that $\ker \tilde{F}  $ is orthogonally complementable in $H_{\mathcal{A}}  $ by \cite[Theorem 2.3.3]{MT}. Hence $T^{m}_{\mid_{\ker \tilde{F}}}  $ is adjointable. Since $T^{m}(\ker \tilde{F})   $ is closed by assumption, again by \cite[Theorem 2.3.3]{MT} we have that $T^{m}(\ker \tilde{F})  $ is orthogonally complementable in $H_{\mathcal{A}}  .$ As $T^{m}(\ker \tilde{F}) \subseteq  Im T^{m} \cap Im (T+F)^{m},$ it is easy to see that $Im T^{m}=T^{m}(\ker \tilde{F}) \oplus N,$ $Im (T+F)^{m}= T^{m} (\ker \tilde{F})  \oplus N^{\prime} $ for some closed submodules $N,N^{\prime}  .$ Now, since $Im \tilde{F}  $ is finitely generated, it follows that $\ker \tilde{F}^{\perp}  $ is finitely generated also, as $ \tilde{F}_{\mid_{\ker \tilde{F}^{\perp}}} $ is an isomorphism onto $Im \tilde{F}   .$ Moreover, $Im T^{m}= T^{m}(\ker \tilde{F})+T^{m}(\ker \tilde{F}^{\perp}),$ $Im (T+F)^{m}=T^{m}(\ker \tilde{F})+ (T+F)^{m}(\ker \tilde{F}^{\perp}) .$ \\
	Let $Q$ denote the orthogonal projection onto $T^{m}(\ker \tilde{F})^{\perp}.$ It is clear then that 
	$N=Q( Im T^{m} ) =Q (T^{m}(\ker \tilde{F}^{\perp}))  $ and $N^{\prime}=Q(Im (T+F)^{m})=Q((T+F)^{m}) (\ker \tilde{F}^{\perp})).$ As $\ker \tilde{F}^{\perp}$ is finitely generated, it follows that $N,N^{\perp}$ are finitely generated also. Since $T_{\mid_{ Im T^{m}}}  $ is $\mathcal{A}$-Fredholm, by previous lemma it follows that $\sqcap T_{\mid_{ T^{m}(\ker \tilde{F})}}  $ is $\mathcal{A}$-Fredholm, where $\sqcap  $ denotes the orthogonal projection onto $  T^{m}(\ker \tilde{F}) $ along $N.$ But, since $T^{m}(\ker \tilde{F})^{\perp}=N \oplus Im {T^{m}}^{\perp}  , $ ($Im {T^{m}}  $ is orthogonally complementable again by \cite[Theorem 2.3.3]{MT}), if we let $P  $ denote the orthogonal projection onto $ T^{m}(\ker \tilde{F}) $ along $T^{m}(\ker \tilde{F})^{\perp}  ,$ then $ PT_{\mid_{ T^{m}(\ker \tilde{F})}} $ is an $\mathcal{A}$-Fredholm operator on $T^{m}(\ker \tilde{F})  ,$ as $PT_{\mid_{ T^{m}(\ker \tilde{F})}}= \sqcap T_{\mid_{ T^{m}(\ker \tilde{F})}} .$ By previous lemma, since $Im T^{m} \cong H_{\mathcal{A}}  $ by assumption, it follows that $index T= index T_{\mid_{ Im T^{m}}}=index  PT_{\mid_{ T^{m}(\ker \tilde{F})}}  .$ Now since $Im T^{m} \cong H_{\mathcal{A}}, Im T^{m}= T^{m}(\ker \tilde{F}) \oplus N   $ and $N$ is finitely generated, by Dupre Filmore theorem it follows easily that $T^{m}(\ker \tilde{F}) \cong H_{\mathcal{A}}.$ Since $PF_{\mid_{ T^{m}(\ker \tilde{F})}} \in K(T^{m}(\ker \tilde{F})) ,$ it follows from \cite[Lemma 2.7.13]{MT} that $P(T+F)_{\mid_{ T^{m}(\ker \tilde{F})}}  $ is an $\mathcal{A}$-Fredholm operator on $T^{m}(\ker \tilde{F})  ,$ and $index PT_{\mid_{ T^{m}(\ker \tilde{F})}} = index P(T+F)_{\mid_{ T^{m}(\ker \tilde{F})}} .$ But $ Im(T+F)^{m}=T^{m}(\ker \tilde{F}) \oplus N^{\prime} $ where $ N^{\prime} $ is finitely generated. Hence $P(T+F)_{\mid_{ T^{m}(\ker \tilde{F})}}=\tilde{\sqcap} T_{\mid_{ T^{m}(\ker \tilde{F})}} $ where $\tilde{\sqcap}  $ denotes the orthogonal projection onto $T^{m}(\ker \tilde{F})  $ along $N^{\prime}  ,$ as 	$(T+F)( T^{m}(\ker \tilde{F})) =(T+F)^{m+1} (\ker \tilde{F}) \subseteq Im (T+F)^{m+1} \subseteq Im (T+F)^{m} .$ In addition, since $N^{\prime}  $ is finitely generated and $T^{m}(\ker \tilde{F})\cong H_{\mathcal{A}}  ,$  by Kasparov stabilization theorem, it follows that $Im (T+F)^{m} \cong H_{\mathcal{A}}  .$ By previous lemma, since $\tilde{\sqcap} T_{\mid_{ T^{m}(\ker \tilde{F})}}  $ is an $\mathcal{A}$-Fredholm operator on $T^{m}(\ker \tilde{F}) , Im (T+F)^{m} \cong H_{\mathcal{A}}  $ and $ N^{\prime}  $ is finitely generated, it follows that $(T+F)_{\mid_{ Im (T+F)^{m}}}  $ is $\mathcal{A}$-Fredholm and $index (T+F)=index  (T+F)_{\mid_{ Im (T+F)^{m}}}= index (\tilde{\sqcap} (T+F)_{\mid_{ T^{m}(\ker \tilde{F})}}  .$
\end{proof} 
\begin{remark}
	Proposition \ref{P04} hold even if $ Im F^{n} $ is not isomorphic to $H_{\mathcal{A}}  $ because $Im F^{n}$ are countably generated being direct summand in $H_{\mathcal{A}}  $  by \cite[Theorem 2.3.3]{MT}  Namely, if $M$ a countably generated Hilbert $C^{*}$-module, then by Kasparov stabilization theorem, $M \oplus H_{\mathcal{A}} \cong H_{\mathcal{A}}  .$ Given an operator $ F \in B^{a}(M) ,$ we may consider the induced operator $F^{\prime} \in B^{a}(M \oplus H_{\mathcal{A}})  $ given by the operator matrix
	$\left\lbrack
	\begin{array}{ll}
	F & 0 \\
	0 & I \\
	\end{array}
	\right \rbrack
	.$
	It is clear then that if $M=M_{1} \tilde \oplus N_{1}‎‎\stackrel{F}{\longrightarrow} M_{2} \tilde \oplus N_{2}=M $ is a decomposition w.r.t. which $F$ has the matrix
	$\left\lbrack
	\begin{array}{ll}
	F_{1} & 0 \\
	0 & F_{4} \\
	\end{array}
	\right \rbrack
	$
	where $F_{1}$ is an isomorphism, then $F^{\prime} $ has the matrix
	$\left\lbrack
	\begin{array}{ll}
	F_{1}^{\prime} & 0 \\
	0 & F_{4}^{\prime} \\
	\end{array}
	\right \rbrack
	$
	w.r.t. the decomposition.\\
	$$ M { \oplus} H_{\mathcal{A}}=(M_{1} \oplus H_{\mathcal{A}}) \tilde{ \oplus} (N_{1} \oplus \lbrace 0 \rbrace ) \stackrel{F^{\prime}}{\longrightarrow} (M_{2} \oplus H_{\mathcal{A}}) \tilde{ \oplus} (N_{2} \oplus \lbrace 0 \rbrace )=M \oplus H_{\mathcal{A}}$$
	where $F_{1}^{\prime}  $ is an isomorphism. It follows then that any semi-Fredholm decomposition for $F$ gives a rise in a natural way to a semi-Fredholm decomposition of $ F^{\prime}.$ Moreover, $F^{\prime}  $ can be viewed as an operator in $B^{a}(H_{\mathcal{A}})  $ as $M \oplus H_{\mathcal{A}} \cong H_{\mathcal{A}}.$ It follows easily then that $index F$ is well defined as $index F^{\prime}  $ is so, (when $F \in \mathcal{M}\Phi(M))  $ and in this case $index F=index F^{\prime} .$ Thus \cite[Theorem 2.7.9]{MT} holds for $F.$ Similarly \cite[Lemma 2.7.11]{MT}, \cite[Lemma 2.16]{I}, \cite[Lemma 2.17]{I} also hold for $F$.
\end{remark}


\begin{thebibliography}{99}
	
	
	
	\bibitem{BS} M. Berkani and M. Sarih, \textit{ON SEMI B-FREDHOLM OPERATORS}, Glasgow Mathematical Journal, Volume 43, Issue 3. May 2001 , pp. 457-465, DOI: https://doi.org/10.1017/S0017089501030075
	
	\bibitem{BM} M. Berkani, \textit{ON SEMI B-FREDHOLM OPERATORS}, PROCEEDINGS OF THE 	AMERICAN MATHEMATICAL SOCIETY, Volume 130, Number \textbf{6}, Pages 1717–1723, S 0002-9939(01)06291-8, Article electronically published on October 17, 2001
	
	\bibitem{DDj2} Dragan S.Djordjevi\'{c}, \textit{On generalized Weyl operators}, PROCEEDINGS OF THE AMERICAN MATHEMATICAL SOCIETY, Volume 130, Number \textbf{1}, Pages 81 ll4, s ooo2-9939(01)0608r-6, April 26,2OO1
		
	\bibitem{I} S. Ivkovi\'{c} , \textit{Semi-Fredholm theory on Hilbert C*-modules}, Banach Journal of Mathematical Analysis, to appear (2019), arXiv:  https://arxiv.org/abs/1906.03319  
	
	\bibitem{S} S. Ivkovi\'{c} , \textit{Spectral semi-Fredholm theory on Hilbert $C^{*}$-modules}, arXiv:  https://arxiv.org/abs/1906.05359 
	
	
	
	\bibitem{IM} Anwar A. Irmatov and Alexandr S. Mishchenko , \textit{On Compact and Fredholm Operators over C*-algebras and a New Topology in the Space of Compact Operators}, J. K-Theory 2 (2008), 329–351 
	
	
	
	
	
	
	\bibitem{KY} Kung Wei Yang, \textit{The generalized Fredholm operators}, Transactions of the American Mathematical Society
	Vol. \textbf{216} (Feb., 1976), pp. 313-326
	
	\bibitem{MF} A. S. Mishchenko, A.T. Fomenko, \textit{The index of eliptic operators over C*-algebras}, Izv. Akad. Nauk SSSR Ser. Mat. \textbf{43} (1979), 831--859; English transl., Math. USSR-Izv.\textbf{15} (1980) 87--112.
	
	\bibitem{MT} V. M. Manuilov, E. V. Troitsky, \textit{Hilbert C*-modules}, In: Translations of Mathematical Monographs. 226, American Mathematical Society, Providence, RI, 2005.
	

	
	
	
	
	
	
	
	
	
	
	
	
	
	
	
	
	
	
	
	
	
	
	
	
\end{thebibliography}
\end{document}